\documentclass[11pt]{amsart}

\usepackage[a4paper,margin=3cm]{geometry}
\usepackage{amsmath,amssymb,amsthm}
\usepackage[T1]{fontenc}
\usepackage{lmodern}
\usepackage{tikz-cd}
\usepackage{mathrsfs}
\usepackage{calligra}
\usepackage[colorlinks=true,linkcolor=blue,citecolor=blue,urlcolor=blue]{hyperref}
\usepackage{xparse}
\usepackage{environ}
\usepackage[babel]{csquotes}
\usepackage{microtype}

\newcommand{\R}{\mathbf{R}}
\newcommand{\Q}{\mathbf{Q}}
\renewcommand{\P}{\mathbf{P}}

\newcommand{\A}{\mathbf{A}}
\newcommand{\K}{\mathbf{K}}
\DeclareMathOperator{\Hom}{\mathscr{H}\text{\kern -3pt {\calligra\large om}}\,}
\DeclareMathOperator{\Spec}{\mathscr{S}\text{\kern -4pt {\calligra\large pec}}\,}

\newcommand*{\DashedArrow}[1][]{\mathbin{\tikz [baseline=-0.25ex,-latex, dashed,#1] \draw [#1] (0pt,0.5ex) -- (1.3em,0.5ex);}}

\renewcommand{\qedsymbol}{\ensuremath{\blacksquare}}

\theoremstyle{plain}
\newtheorem{thm}{Theorem}[section]
\newtheorem{lem}[thm]{Lemma}
\newtheorem{cor}[thm]{Corollary}

\theoremstyle{definition}
\newtheorem{defn}[thm]{Definition}
\newtheorem{ex}[thm]{Example}
\newtheorem{prop}[thm]{Proposition}

\theoremstyle{remark}

\newenvironment{proofA}{\vspace{0.2cm}\paragraph{\bf\textit{Proof  of Theorem \ref{thm:p1z-cox}}}}{\hfill \qedsymbol \medskip}

\newenvironment{proofB}{\vspace{0.2cm}\paragraph{\bf\textit{Proof  of Theorem \ref{thm:birgeom}}}}{\hfill \qedsymbol \medskip}

\makeatletter
\renewenvironment{proof}[1][\proofname]{\par
  \pushQED{\qed}%
  \normalfont \topsep6\p@\@plus6\p@\relax
  \trivlist
  \item[%
    \hskip\labelsep
    \itshape\bfseries 
    #1%
    \@addpunct{.}
  ]\ignorespaces
}{%
  \popQED\endtrivlist\@endpefalse
}
\let\qed\relax 
\DeclareRobustCommand{\qed}{%
  \ifmmode \mathqed
  \else
    \leavevmode\unskip\penalty\@M\hbox{}\nobreak\hspace{.5em minus .1em}
    \hbox{\qedsymbol}%
  \fi
}
\makeatother

\ExplSyntaxOn
\NewEnviron{pmatrixT}
 {
  \marine_transpose:V \BODY
 }
\int_new:N \l_marine_transpose_row_int
\int_new:N \l_marine_transpose_col_int
\seq_new:N \l_marine_transpose_rows_seq
\seq_new:N \l_marine_transpose_arow_seq
\prop_new:N \l_marine_transpose_matrix_prop
\tl_new:N \l_marine_transpose_last_tl
\tl_new:N \l_marine_transpose_body_tl

\cs_new_protected:Nn \marine_transpose:n
 {
  \seq_set_split:Nnn \l_marine_transpose_rows_seq { \\ } { #1 }
  \int_zero:N \l_marine_transpose_row_int
  \prop_clear:N \l_marine_transpose_matrix_prop
  \seq_map_inline:Nn \l_marine_transpose_rows_seq
   {
    \int_incr:N \l_marine_transpose_row_int
    \int_zero:N \l_marine_transpose_col_int
    \seq_set_split:Nnn \l_marine_transpose_arow_seq { & } { ##1 }
    \seq_map_inline:Nn \l_marine_transpose_arow_seq
     {
      \int_incr:N \l_marine_transpose_col_int
      \prop_put:Nxn \l_marine_transpose_matrix_prop
       {
        \int_to_arabic:n { \l_marine_transpose_row_int }
        ,
        \int_to_arabic:n { \l_marine_transpose_col_int }
       }
       { ####1 }
     }
   }
   \tl_clear:N \l_marine_transpose_body_tl
   \int_step_inline:nnnn { 1 } { 1 } { \l_marine_transpose_col_int }
    {
     \int_step_inline:nnnn { 1 } { 1 } { \l_marine_transpose_row_int }
      {
       \tl_put_right:Nx \l_marine_transpose_body_tl
        {
         \prop_item:Nn \l_marine_transpose_matrix_prop { ####1,##1 }
         \int_compare:nF { ####1 = \l_marine_transpose_row_int } { & }
        }
      }
     \tl_put_right:Nn \l_marine_transpose_body_tl { \\ }
    }
   \begin{pmatrix}
   \l_marine_transpose_body_tl
   \end{pmatrix}
 }
\cs_generate_variant:Nn \marine_transpose:n { V }
\cs_generate_variant:Nn \prop_put:Nnn { Nx }
\ExplSyntaxOff

\title{Remarks on hypersurfaces in $\P^1\times Z$}
\author[F.\ A.\ Denisi]{Francesco Antonio Denisi}
\address{Francesco Antonio Denisi, Fachrichtung Mathematik, Campus, Gebäude E2 4, Universität des Saarlandes, 66123 Saarbrücken, Deutschland}
\email{denisi@math.uni-sb.de}

\author[A.\ Laface]{Antonio Laface}
\address{Universidad de Concepción, Concepción, Chile}
\email{alaface@udec.cl}

\date{\today}

\begin{document}

\maketitle

\begin{abstract}
We study the birational geometry of hypersurfaces in projective varieties of the form $\mathbf{P}^1 \times Z$, where $Z$ satisfies mild assumptions. Building on recent results of Herrera--Laface--Ugaglia, we study their Cox rings—when finitely generated—in terms of the Cox ring of $Z$. In particular, we obtain a complete picture when $Z$ is a Fano variety of dimension at least 3,  and with class group of rank 1.
\end{abstract}

\section*{Acknowledgements}
 Denisi is supported by the Deutsche Forschungsgemeinschaft (DFG, German Research Foundation) – Project ID 286237555 (TRR 195).
 Laface is partially supported by Proyecto FONDECYT Regular n.~1230287.
\section{Introduction}

Given a projective variety $X$, defined over an algebraically closed field of characteristic 0, and a hypersurface $Y \subset X$, it is a natural problem to study how the birational geometry of $Y$ compares with that of $X$. For example, assuming that $X$ is a Mori dream space, one may ask for conditions under which $Y$ is itself a Mori dream space.
\medskip

Mori dream spaces were introduced by Hu and Keel in \cite{HK00} and play an important role in birational geometry, as they exhibit ideal behaviour under the minimal model program (MMP). Indeed, it turns out that any MMP for any divisor can be carried out. Also, any nef divisor is semiample by definition. Due to these strict requirements, finding examples of Mori dream spaces is both interesting and challenging.
\medskip

By definition, a Mori dream space $Y$ has a \enquote{finite birational geometry}, and hence determining all possible maps $Y \DashedArrow Z$ arising in the study of the Mori theory of $Y$, together with the relevant associated cones of divisors in the Néron--Severi space (which in this case are all rational polyhedral), is both a natural and challenging problem. Moreover, associated with $Y$ is a finitely generated algebra over the ground field, called the \emph{Cox ring} of the Mori dream space. Whenever one has such an algebraic object at hand, it is naturally of interest to provide an explicit presentation of it.
\medskip

In this note, we address these problems for hypersurfaces in ambient varieties of the form $\mathbf{P}^1 \times Z$, recovering in some specific cases existing results in the literature (see, for example, \cite{Ito14}, \cite{Ottem15}, \cite{Denisi24}, \cite{hlu}, and \cite{PIS25}).

\begin{thm}\label{thm:birgeom}
Let $Z$ be a normal, $\Q$-factorial and Cohen--Macaulay projective variety of dimension $m\geq 3$, and with divisor class group of rank 1. Set $X:=\P^1\times Z$, and let $L=\mathscr{O}_{\mathbf{P}^1}(d)\boxtimes L'$ be an ample and globally generated line bundle, with $1\leq d\leq m$. Let $Y$ be a general member of $|L|$, $H_1$ the class in $N^1(Y)_{\R}$ of the restriction to $Y$ of any divisor in $|p_1^*\mathscr{O}_{\P^1}(1)|$, and $H_2$ the class of the restriction of any divisor in $|p_2^*L'|$ to $Y$. Then we have:
\begin{enumerate}
 \item If $1<d<m$, then
\[
\mathrm{Eff}(Y)=\mathrm{Mov}(Y)=\R_{\geq 0}H_1+\R_{\geq 0}(H_2-H_1),
\] 
\[
\mathrm{Nef}(Y)=\R_{\geq 0}H_1+\R_{\geq 0}H_2.
\]
In this case, $Y$ admits (up to isomorphism) a unique small $\Q$-factorial modification, a fibration onto $\P^1$ and a dominant rational fibration to $\P^{d-1}$.

\item If $d=1$, then
\[
\mathrm{Eff}(Y)=\R_{\geq 0}H_1+\R_{\geq 0}(H_2-H_1),
\]
\[
\mathrm{Nef}(Y)=\mathrm{Mov}(Y)=\R_{\geq 0}H_1+\R_{\geq 0}H_2.
\]
In this case, $Y$ admits a divisorial contraction and a fibration onto $\P^1$.

\item If $d=m$, then
\[
\mathrm{Eff}(Y)=\mathrm{Mov}(Y)=\mathrm{Nef}(Y)=\R_{\geq 0}H_1+\R_{\geq 0}(H_2-H_1).
\] 
In this case, $Y$ has a fibration onto $\P^{d-1}$ and a fibration onto $\P^1$.
 \end{enumerate}

 Moreover, if $\mathrm{Pic}(Z)$ is finitely generated, the variety $Y$ is a Mori dream space.
\end{thm}

Motivated by the fact that the birational geometry of the hypersurfaces considered in Theorem \ref{thm:birgeom} does not depend on the choice of the factor $Z$, and that the Cox ring of a projective variety encodes information about the birational geometry of the variety, one may hope that the Cox rings of such hypersurfaces can likewise be described in terms of the Cox ring of $Z$. Very recent results of Herrera--Laface--Ugaglia (see \cite{hlu}) provide the appropriate strategy to address this problem and yield under mild assumptions a positive answer to this expectation.

\begin{thm}\label{thm:p1z-cox}
Let $Z$ be a normal projective variety of dimension $m\geq 3$ with finitely generated divisor class group of rank one, and let $R:=\mathcal R(Z)$ be its Cox ring.
Set $X:=\P^1\times Z$ and let
$Y\subseteq X$ be a normal hypersurface defined by the
equation 
\[
f=
f_0T_1^d-f_1T_1^{d-1}T_2+\cdots+(-1)^df_dT_2^d,
\tag{1}
\] 
where $f_0,\dots,f_d\in R$ are homogeneous elements of
the same degree. Suppose that:
\begin{enumerate}
\item The divisor $Y$ in $X$ is Cartier;
\item $V(f)\setminus V(T_1,T_2)\subseteq \mathbf A^2\times\overline Z$ is
normal;
\item the pullback
$\iota^*\colon \operatorname{Cl}(\mathbf P^1\times Z)\to \operatorname{Cl}(Y)$
is an isomorphism;
\item $f_0,\dots,f_d$ is a regular sequence in $R$.
\end{enumerate}
Then the Cox ring of $Y$ is
\[
\mathcal R(Y)\cong
\frac{
R[T_1,T_2,S_1,\dots,S_d]
}{
\left\langle
f_0+T_2S_1,\,
f_1+T_1S_1+T_2S_2,\,
\dots,\,
f_{d-1}+T_1S_{d-1}+T_2S_d,\,
f_d+T_1S_d
\right\rangle
}.
\]
The grading on the right-hand side is induced by
$\operatorname{Cl}(Y)$ via the isomorphism $\iota^*$; in particular, one has
$\deg S_i=\deg f_i-\deg T_1$ for every $i$.
\end{thm}

We observe that Theorem \ref{thm:p1z-cox} is stronger than Theorem \ref{thm:birgeom}, provided that item (2) in the statement is satisfied. Indeed, whereas Theorem \ref{thm:birgeom} requires the ambient variety to be Cohen--Macaulay and the hypersurface to be sufficiently general to determine the birational geometry explicitly, Theorem \ref{thm:p1z-cox} allows us to dispense with the Cohen--Macaulay assumption on the ambient variety and imposes significantly weaker hypotheses on the hypersurface. The corollary below provides a sufficient condition such that all the assumptions in Theorem \ref{thm:p1z-cox} are satisfied.

\begin{cor}\label{cor:general-p1z-cox}
Let $Z$ be a normal projective variety of dimension $m\geq 3$ with finitely generated divisor class group of rank one, and let $R:=\mathcal R(Z)$ be its Cox ring.
Assume that $R$ is Cohen--Macaulay. 
Set $X:=\mathbf P^1\times Z$.
Let $L$ be an ample and base-point free line bundle on $X$, and $Y\in |L|$ a general normal hypersurface. Write the Cox
equation of $Y$ as
\[
f=
f_0T_1^d-f_1T_1^{d-1}T_2+\cdots+(-1)^df_dT_2^d,
\]
where $f_0,\dots,f_d\in R$ are homogeneous elements. Assume that
$f_0,\dots,f_d$ is a regular sequence in $R$. Then
\[
\mathcal R(Y)\cong
\frac{
R[T_1,T_2,S_1,\dots,S_d]
}{
\left\langle
f_0+T_2S_1,\,
f_1+T_1S_1+T_2S_2,\,
\dots,\,
f_{d-1}+T_1S_{d-1}+T_2S_d,\,
f_d+T_1S_d
\right\rangle }.
\]
The grading on the right-hand side is induced by
$\operatorname{Cl}(Y)$ via the restriction isomorphism
$\operatorname{Cl}(X)\simeq\operatorname{Cl}(Y)$.
\end{cor}

It turns out that the hypotheses of Corollary \ref{cor:general-p1z-cox} are actually satisfied by all Fano varieties over any algebraically closed field of characteristic 0 with a class group of rank 1, thus giving an infinite series of new examples for which Theorem \ref{thm:p1z-cox} holds.

\begin{cor}\label{cor:Fano}
Let $Z$ be a Fano variety of dimension $m\geq 3$ with divisor class group of rank one, and $R:=\mathcal R(Z)$ its Cox ring. Let $L$ be an ample and base-point free line bundle on $X:=\P^1\times Z$, and $Y\in |L|$ a general hypersurface. Write the Cox
equation of $Y$ as
\[
f=
f_0T_1^d-f_1T_1^{d-1}T_2+\cdots+(-1)^df_dT_2^d,
\]
where $f_0,\dots,f_d\in R$ are homogeneous elements, with 
$d\leq m$.
Then
\[
\mathcal R(Y)\cong
\frac{
R[T_1,T_2,S_1,\dots,S_d]
}{
\left\langle
f_0+T_2S_1,\,
f_1+T_1S_1+T_2S_2,\,
\dots,\,
f_{d-1}+T_1S_{d-1}+T_2S_d,\,
f_d+T_1S_d
\right\rangle }.
\]
\end{cor}

It is worth observing that in the three statements above, although the divisor class group has rank one (i.e., its free part is isomorphic to $\mathbf Z$), it may still have nontrivial torsion. Moreover, all the varieties involved are $\Q$-factorial.

\section{Preliminaries}

Throughout, we work over an algebraically closed field $\mathbf K$ of
characteristic zero. A variety will be an integral separated $\mathbf{K}$-scheme of finite type.
\medskip

Given any projective variety $X$, the Néron--Severi space of $X$ is $N^1(X)_{\R}=\mathrm{Div}(X)_{\R}/{\equiv}$, where $\mathrm{Div}(X)_{\R}$ is the group of Cartier divisors with real coefficients, and $\equiv$ is the numerical equivalence relation, that is, $D_1\equiv D_2$ if and only if $D_1\cdot C=D_2\cdot C$, for every irreducible curve $C$.
\medskip

The movable cone $\mathrm{Mov}(X)$ of $X$ is the cone in the Néron--Severi space spanned by movable Cartier divisor classes, that is, classes of integral Cartier divisors whose stable base loci have codimension at least $2$. The effective cone $\mathrm{Eff}(X)$ is the cone spanned by effective Cartier divisor classes, while the nef cone $\mathrm{Nef}(X)$ is the cone spanned by nef divisor classes, namely divisors $D$ satisfying $D\cdot C \geq 0$ for every irreducible curve $C$.

\subsection{Reminders on vector bundles and degeneracy loci}
In this subsection, we recall some basic notions and results concerning vector bundles on projective varieties.

\begin{defn}
Let $X$ be a projective variety and $E$ a vector bundle (i.e.\ a locally free $\mathscr{O}_X$-module). We define $V(E):=\Spec(\mathrm{Sym}(E))$, i.e.\ $V(E)\to X$ is the geometric vector bundle associated with $E$. 
\end{defn}

\begin{defn}
Let $X$ be a projective variety, and $\varphi \colon E \to F$ a morphism of vector bundles on $X$, with $\mathrm{rk}(E)\geq \mathrm{rk}(F)$. We know that $\varphi \in H^0(X,E^{\vee}\otimes F)$. Let $k$ be a positive integer, and consider 
\[
\wedge^{k+1} \varphi\in H^0(X,\wedge^{k+1}{E^{\vee}}\otimes \wedge^{k+1}{F}).
\]
The section $\wedge^{k+1} \varphi$ defines a morphism of vector bundles 
\[
h \colon \mathscr{O}_X \to \wedge^{k+1}{E^{\vee}}\otimes \wedge^{k+1}{F}.
\]
 If we dualise $h$, we obtain 
 \[
 h^{\vee} \colon (\wedge^{k+1}{E^{\vee}}\otimes \wedge^{k+1}{F})^{\vee} \to \mathscr{O}_X.
 \]
The $k$-th degeneracy locus of $\varphi$ is the closed subscheme $D_k(\varphi)$ defined by the quasi-coherent sheaf of ideals $\mathscr{I}=\mathrm{im}(h^{\vee})$. Set-theoretically, $D_k(\varphi)$ consists of the (closed) points $x$ of $X$ where the induced morphism $\varphi_x \colon E_x/\mathfrak{m}_x E_x \to F_x/\mathfrak{m}_x F_x$ between the fibres has rank smaller than or equal to $k$.
\end{defn}

Recall that a Noetherian scheme $X$ is Cohen--Macaulay if for any closed point $x$ the ring $\mathscr{O}_{X,x}$ is Cohen--Macaulay, i.e.\ $\mathrm{depth}(\mathscr{O}_{X,x})=\mathrm{dim}(\mathscr{O}_{X,x})$. If $M^{m\times n}(\K)\cong \A^{mn}$ is the affine space of matrices of size $m\times n$ over $\K$, we denote by $M_k$ the affine closed subscheme of $\A^{mn}$ parametrising the matrices $M$ of rank $\mathrm{rk
}(M)\leq k$. The scheme $M_k$ is a Cohen--Macaulay affine variety (see for example \cite[ (3.1) Theorem]{ACGH13}).
\medskip

We recall the following useful lemma, for which a proof can be found in \cite[Lemma 3.7]{Denisi24}.

\begin{lem}\label{lem:CMvar}
Let $X$ be a Cohen--Macaulay variety, $E$, $F$, two vector bundles, and suppose that the vector bundle $\Hom(E,F)$ is globally generated. Let $s_0\colon X\hookrightarrow V(\Hom(\wedge^kE,\wedge^kF))$ be the zero section of  the vector bundle $V(\Hom(\wedge^kE,\wedge^kF))\to X$, and $k$ any positive integer. Furthermore, let $\Sigma$ be the inverse image scheme of $s_0(X)$ via the natural morphism $V(\Hom(E,F))\to V(\Hom(\wedge^kE,\wedge^kF))$, and denote by $V$ the inverse image scheme of $\Sigma$ via the natural morphism of vector bundles $X\times H^0(E^{\vee}\otimes F) \to V(\Hom(E,F))$. Then $V$ and $\Sigma$ are Cohen--Macaulay varieties.
\end{lem}

\subsection{Cox Sheaves and Cox rings}
We begin by recalling the definition of a Cox sheaf. Let $X$ be a normal variety such that $\Gamma(X,\mathcal{O}^*) = \K^*$. Denote by ${\rm WDiv}(X)$ the group of Weil divisors of $X$.
The {\em divisor class group} of $X$ is defined as  
\[
 {\rm Cl}(X) :={\rm WDiv}(X)/\sim,
\]  
where $\sim$ is the linear equivalence relation. We can associate with any subgroup $\Gamma \subseteq {\rm WDiv}(X)$ the following {\em sheaf of divisorial algebras}:  
\[
 \mathcal{S} := \bigoplus_{D \in \Gamma} \mathcal{O}_X(D).
\]  
Suppose now that the quotient map $\Gamma \to {\rm Cl}(X)$ is surjective, and let $\Gamma^0$ be its kernel.  
Fix a homomorphism $\chi \colon \Gamma^0 \to \mathbf{K}(X)^*$ such that ${\rm div}(\chi(D)) = D$ for any $D \in \Gamma^0$.  
Let $\mathscr{I} \subseteq \mathcal{S}$ be the ideal sheaf locally generated by elements of the form $1 - \chi(D)$, where $D$ ranges over $\Gamma^0$.  
From now on, assume that ${\rm Cl}(X)$ is finitely generated.  
The {\em Cox sheaf} is the ${\rm Cl}(X)$-graded sheaf of algebras defined as  
\[
 \mathcal{R}_X := \mathcal{S} / \mathscr{I}.
\]  
Choosing a different $\Gamma$ or a different $\chi$ does not change the isomorphism class of the Cox sheaf.  
If $X$ is $\Q$-factorial, then the Cox sheaf is locally of finite type~\cite[\S 1.6.1]{adhl}.  
In this case, we denote its relative spectrum by  
\[
 \widehat{X} := {\rm Spec}_X \mathcal{R}_X.
\]  
It can be shown that $\widehat{X}$ is an irreducible, normal, quasiaffine variety~\cite[Construction 1.6.1.3]{adhl}.

\begin{defn}
The {\em Cox ring} of $X$ is defined as the ring of global sections of the Cox sheaf, i.e.\  
\[
 \mathcal{R}(X) := \Gamma(X, \mathcal{R}_X).
\]  
\end{defn}

\subsection{Notions in birational geometry}

In this subsection, we recall a few definitions from birational geometry.

\begin{defn}
A small birational map between two projective varieties $X,Y$ is a birational map $f\colon X\DashedArrow Y$ which is an isomorphism in codimension 1, that is, there exist open subsets $U\subset X$ and $V \subset Y$ such that $f_{|U}\colon U \to f(U)=V$ is an isomorphism, and $\mathrm{codim}_XU\geq 2$, $\mathrm{codim}_YV\geq 2$. If $X$ and $Y$ are normal and $\Q$-factorial, we say that $f$ is a small $\Q$-factorial modification.
\end{defn}

The notion of a small $\mathbf{Q}$-factorial modification lies at the heart of the definition of Mori dream space.

\begin{defn}
A normal and $\Q$-factorial projective variety $Y$ is called a Mori dream space if the following conditions are satisfied:
\begin{enumerate}
\item $\mathrm{Pic}(Y)$ is finitely generated,
\item  the nef cone $\mathrm{Nef}(Y)$ is generated by finitely many Cartier divisors which are semiample,
\item  there exist finitely many small $\Q$-factorial modifications $f_i\colon Y \DashedArrow Y'_i$ such that $Y'_i$ satisfies the first two items, for each $i$, and
\[\mathrm{Mov}(Y)=\bigcup_{i} f_i^*(\mathrm{Nef}(Y'_i)).\]
\end{enumerate}
\end{defn}

Thanks to the results of Birkar--Cascini--Hacon--McKernan (\cite{BCHM10}), one of the most important classes of Mori dream spaces is that of Fano varieties.

\begin{defn}
A Fano variety is a normal, $\Q$-factorial projective variety with klt singularities, such that the anticanonical divisor is ample.
\end{defn}

The definition of a Mori dream space translates algebraically, namely, a normal and $\Q$-factorial projective variety $X$ is a Mori dream space if and only if the Cox ring $\mathcal{R}(X)$ is a finitely generated algebra over the ground field, by results of Hu--Keel (\cite{HK00}).

\section{The birational geometry of the hypersurfaces}

We start this section by providing a general result on the existence of small birational maps for hypersurfaces $Y$ in $\P^1 \times Z$. In what follows, $p_1$ and $p_2$ denote the projections, $H_1$ denotes the class in $N^1(Y)_{\R}$ of the restriction to $Y$ of any divisor in $|p_1^*\mathscr{O}_{\P^1}(1)|$, and $H_2$ the class of the restriction of any divisor in $|p_2^*L'|$ to $Y$. The construction was inspired by \cite{Ottem15}.

\begin{prop}\label{prop:birgeom}
Let $Z$ be a Cohen--Macaulay projective variety of dimension $m \geq 3$. Set $X:=\P^1\times Z$, and let $L=\mathscr{O}_{\mathbf{P}^1}(d)\boxtimes L'$ be a big and globally generated line bundle, with $2\leq d\leq m$. 
Then the general member $Y$ of $|L|$ is Cohen--Macaulay, and admits a small birational map $Y\DashedArrow Y'$ to another Cohen--Macaulay projective variety $Y'$. Moreover, if $Y$ is normal, then $Y'$ is normal as well.
\end{prop}
\begin{proof}
\vspace{0.1cm}

First of all, we observe that the general member of $|L|$ is a Cohen--Macaulay variety because the ambient variety is, and $L$ is an invertible sheaf. Since $L$ is big and globally generated, then $L'$ is, and the zero locus of a general section of $\oplus_{i=0}^d L'$ has pure dimension $m-d-1$. We want to show that $Y$ admits a small birational map $Y\DashedArrow Y'$. Let $T_1, T_2$ be the coordinates of $\P^1$, whose respective pull-back generate the line bundle $p_1^*\mathscr{O}_{\P^1}(1)$. Then we can write
\[
Y=\{T_1^df_0+T_1^{d-1}T_2f_1+\cdots+T_1T_2^{d-1}f_{d-1}+T_2^df_d=0\},
\]
where $\{f_i\}_{i=0}^d$ is a set of general global sections of $L'$. We observe that $Y$ is the degeneracy locus of a morphism of vector bundles on $\P^1 \times Z$. For convenience, we set $$E=\mathscr{O}_X^{d+1}, \; F=p_1^*\mathscr{O}_{\P^1}(1)^{\oplus d} \oplus p_2^*L', \; E'=\mathscr{O}_{X'}^{d+1},\; F'=p_1'^*\mathscr{O}_{\P^{d-1}}(1)^{\oplus 2}\oplus p_2'^*L'.$$  In particular, $Y$ is the first degeneracy locus of the morphism 
\[
M\colon E \to F,
\] 
defined by the matrix
\[
M=\begin{pmatrix}
    T_2   &   0    & \cdots &  0  & f_0 \\
   -T_1   &  T_2   & \ddots & \vdots  & f_1 \\
     0    &  -T_1   & \ddots &   0   &   \vdots      \\
  \vdots  & \vdots & \ddots&  T_2   & f_{d-1} \\
     0    & 0     & \cdots & -T_1  & f_d
\end{pmatrix}.
\]
Now, we observe that, unless $T_1=T_2=0$, the rank of $M$ decreases at most by 1, and so $\mathrm{ker}(M)$ is at most 1-dimensional. Let $U$ be the open subset of $Y$ consisting of points $(t_1,t_2,z)$ for which there exists an $i$ such that $f_i$ does not vanish. For points in $U$ we have that $\mathrm{ker}(M)$ is one-dimensional and generated by a unique element of the form $(t_1',\dots,t_d',1)^t$. We observe that the point $(t_1',\dots,t_d',z)$ satisfies the equations defined by the maximal minors of the matrix
\[
N=\begin{pmatrix}
    0     &    T_1'     &  f_0    \\
   -T_1'   &    T_2'     &  f_1    \\
  \vdots  &   \vdots   & \vdots  \\
-T_{d-1}'  &    T_d'     & f_{d-1} \\
     -T_d' &     0      & f_{d} 
\end{pmatrix}.
\]
Let us denote by $Y'$ the closed subscheme of $X'=\P^{d-1}\times Z$ defined by the vanishing of the maximal minors of $N$. The scheme $Y'$ is the first degeneracy locus of the morphism of vector bundles
\[
N\colon E' \to F'
\]
defined by the matrix $N$. Then we have a well-defined rational map $f\colon Y\dashrightarrow Y'$ of schemes (defined at points in $U$) which, set theoretically, is defined as  
\begin{equation}\label{birationalmap}
f\colon (t_1,t_2,y_0,\dots,y_n) \mapsto (t_1',\dots,t_d',y_0,\dots,y_d). 
\end{equation}
In this way, we obtain a bijective correspondence between the matrices of the form of $M$, and the matrices of the form of $N$. These form in $H^0(X,E^{\vee}\otimes F)$ (resp.\ $H^0(X',{E'}^{\vee}\otimes {F'})$) an affine subspace $\A^l$, where $l=(d+1)\cdot h^0(Z,L')$. 
\medskip

We now show that when $Y$ is sufficiently general, the scheme $Y'$ is a Cohen--Macaulay variety. The indeterminacy locus of $f$ is contained in the set $\{f_0=\dots=f_d=0\}\subset Y$. Since $Y$ is general, we know that the indeterminacy locus of $f$ has codimension at least $d$ in $Y$. We observe that either $Y'$ is irreducible, or has two irreducible components, one of which coincides with the closure of the subset of points for which there exists one $i$ such that $f_i$ does not vanish, and has necessarily dimension $m$. In the second case, the other irreducible component would have dimension $m-2$ (and so codimension $d+1$) and coincides set theoretically with the locus $\{f_0=\dots=f_d=0\}\subset \P^{d-1}\times Z$. But the codimension of $Y'$ in $X':=\P^{d-1}\times Z$ is smaller than or equal to $d-1$. Indeed, following notation of Lemma \ref{lem:CMvar}, the scheme $Y'$ is set-theoretically the intersection of $X'$ and $\Sigma$ in $$V:=V\left[\mathscr{O}_{X'}^{d+1}\otimes \left(p_1'^*\mathscr{O}_{\P^{d-1}}(1)^{\oplus 2}\oplus p_2'^*L'\right)\right],$$ and $X$ is embedded in $V$ via the section $N$. This embedding is a regular immersion since it is a section of a smooth morphism (\cite[Lemma 31.22.8]{SP24}). Then any irreducible component of $\Sigma \cap X'$ has codimension at most $d-1$ in $X'$, by \cite[Lemma 43.13.2]{SP24} (see also \cite[Lemma 2.7]{Ott95}, or  \cite[Theorem 14.4(b)]{Fulton98}). This implies that $Y'$ is irreducible, and hence Cohen--Macaulay, as it has the expected codimension (see \cite[Theorem 14.4(c)]{Fulton98}). In particular, to show that $Y'$ is reduced, we only need to find a closed point $y'$ at which $Y'$ is reduced.
 \medskip
 
We denote by $V$ the incidence correspondence 
\[V=\{(x',\phi)\in X' \times H^0(X,E'^{\vee}\otimes F') \;|\; \text{rk}(\phi_{x'})\leq 2 \}.\]
By Lemma \ref{lem:CMvar} the scheme $V$ is a closed Cohen--Macaulay subscheme of $$X' \times H^0(X,E'^{\vee}\otimes F')$$ of codimension $d-1$.
\medskip

Since $Y$ is general, we may assume that the sections $f_0,\dots,f_{d-2}$ define a reduced subscheme of $Z$, and form a regular sequence. Consider the projection $p_1'\colon Y'\to \P^{d-1}$ and consider the locus $W$ of $\P^{d-1}$ where the fibres of  $p_1'$  have the expected (pure) dimension $m-d+1$. This is an open dense subset of $\P^{d-1}$, by (\cite[Corollary 13.1.5]{EGA4.3}). Since $Y'$ is Cohen--Macaulay, by the miracle flatness Theorem the morphism ${p_1'}_{|{p_1'}^{-1}(W)}\colon {p_1'}^{-1}(W) \to W$ is flat. In particular, the point $t'=(0\colon \dots \colon 0 \colon 1)\in \P^{d-1}$ belongs to $W$, by our choice of the $f_i$. It follows that $p_1'$ is flat at any point of the scheme-theoretic fibre $W_{t'}=\P^{d-1}\times_{Z} \mathrm{Spec}(k(t'))$. But by our choice of $f_0,\dots,f_{d-2}$, the closed subscheme $W_{t'}$ is reduced. This implies that $Y'$ contains a reduced point (see\ \cite[Theorem 23.9]{Matsumura87}), and hence is everywhere reduced as $Y'$ has no embedded points. It follows that the scheme $Y'$ is a Cohen--Macaulay variety.
\medskip

Now observe that if $(t_1',\dots,t_d')\neq 0$, the rank of $N$ decreases at most by 1. If we denote by $U'$ the open subset of $Y'$ consisting of points such that there exists one $i$ for which $f_i$ does not vanish, arguing as above, we obtain a well-defined map $g\colon Y' \DashedArrow Y$, $$g\colon (t_1',\dots,t_d',y_0,\dots,y_n) \mapsto (t_1,t_2,y_0,\dots,y_d),$$ where $(t_0,t_1,1)^t$ is a generator of $\mathrm{ker}(N)$ evaluated at $(t_1',\dots,t_d',y_0,\dots,y_n)$. It is easy to check that $f$ and $g$ are mutually inverse, and hence we conclude that $f$ is a small birational map.
\medskip

Suppose now that $Y$ is normal. Since the map $Y\DashedArrow Y'$ is small, it follows that $Y'$ is regular in codimension one, and hence satisfies Serre's conditions $R_1$. Clearly $Y'$ satisfies condition $S_2$ as well, as it is Cohen--Macaulay. Then Serre's criterion gives that $Y'$ is a normal projective variety.
\end{proof}

\begin{cor}\label{cor}
Let $Z$ be a normal, Cohen--Macaulay and $\Q$-factorial projective variety of Picard rank 1. Let $L$ be an ample and globally generated line bundle on $\P^1\times Z$. Then either the general hypersurface $Y\in |L|$ admits a unique small $\Q$-factorial modification, which is an isomorphism if $d=m$, or a unique divisorial contraction. Moreover 

\begin{enumerate}

\item If $1<d<m$, we have
\[
\mathrm{Eff}(Y)=\mathrm{Mov}(Y)=\R_{\geq 0}H_1+\R_{\geq 0}(H_2-H_1);
\] 
\item If $d=m$, we have
\[
\mathrm{Eff}(Y)=\mathrm{Nef}(Y)=\R_{\geq 0}H_1+\R_{\geq 0}(H_2-H_1);
\] 
\item If $d=1$, we have $$\mathrm{Eff}(Y)=\R_{\geq 0}H_1+\R_{\geq 0}(H_2-H_1)$$ and $$\mathrm{Mov}(Y)=\mathrm{Nef}(Y)=\R_{\geq 0}H_1+\R_{\geq 0}H_2.$$
\end{enumerate}
To conclude, if $\mathrm{Pic}(Z)$ is finitely generated, then $Y$ is a Mori dream space.
\end{cor}
\begin{proof}
First of all, we observe that since $L$ is ample and globally generated, the Picard number of the general hypersurface $Y$ is $2$, by the Grothendieck--Lefschetz Theorem (\cite[Theorem 1]{RS06}). Consider the small $\Q$-factorial modification $f\colon Y \dashrightarrow Y'$ constructed in Proposition \ref{prop:birgeom}. We compute $f_*(H_1)$ and $f_*(H_2)$. The interesting case is $f_*(H_1)$. To compute it, pick the divisor $D:=\{T_1=0\} \cap Y$, whose image via $f$ is contained in the set of points of $Y'$ satisfying the maximal minors of the matrix
\[
N=\begin{pmatrix}
    0     &    T_1'     &  f_0    \\
   -T_1'   &    T_2'     &  f_1    \\
  \vdots  &   \vdots   & \vdots  \\
-T_{d-1}'  &    T_d'     & f_{d-1} \\
     -T_d' &     0      & 0 
\end{pmatrix}.
\]
Since we started from points of $Y$ having $T_1=0$, it follows that the image of $D$ via $f$ is the zero locus of a section of $[\mathscr{O}_{\P^{d-1}}(-1)\boxtimes L']_{|Y'}$, and so $f_*(H_1)=H_2'-H_1'$. It is straightforward to check that $f_*(H_2)=H_2'$. Note that $H_2-H_1$ is not big, for example, because $f_*(H_2-H_1)=H'_1$ is not big in $Y'$. Then the ray $\R_{\geq 0}(H_2-H_1)$ is extremal in $\mathrm{Eff}(Y)$. Now, the Picard number of $Y$ is $2$, and so $\mathrm{Eff}(Y) =\R_{\geq 0}(H_2-H_1)+\R_{\geq 0}H_1$. 
\medskip

\underline{If $1<d<m$}, by the discussion above, it is clear that $\mathrm{Nef}(Y)=\R_{\geq 0} H_2+ \R_{\geq 0} H_1$, and $\mathrm{Mov}(Y)=\mathrm{Eff}(Y)$.
\medskip

\underline{If $d=m$}, the map $f \colon Y \DashedArrow Y'$ is an isomorphism, and hence the Cartier divisor $H_2-H_1$ is nef, and induces the fibration $Y \to \P^{d-1}$. It is clear then that in this case $\mathrm{Eff}(Y)=\mathrm{Mov}(Y)=\mathrm{Nef}(Y)=\R_{\geq 0} (H_2-H_1)+ \R_{\geq 0} H_1$.
\medskip

\underline{If $d=1$}, the variety $Y$ is the blow-up of $Z$ along the ideal sheaf $\mathscr{I}$ generated by $f_0,f_1$, and the blowing-up morphism is the divisorial contraction $p_2 \colon Y \to Z$. The exceptional divisor is given by the--up to scalar--unique section of $[\mathscr{O}_{\P^1}(-1)\boxtimes L']_{|Y}$. Indeed, an equation for $Y$ is $\{T_1f_0+T_2f_1=0\}$. Then, if $f_0=0$, either $T_2=0$, or $f_1=0$. Taking out $\{T_2=0\}$ (which means taking a section of $[\mathscr{O}_{\P^1}(-1)\boxtimes L']_{|Y}$), gives exactly $\{f_0=f_1=0\}$, i.e.\ the exceptional divisor. Moreover, $\mathrm{Mov}(Y)=\mathrm{Nef}(Y)=\R_{\geq 0} H_2+ \R_{\geq 0} H_1$.
\medskip

Assuming now that $\mathrm{Pic}(Z)$ is finitely generated, we observe that any nef divisor on $Y$ is semiample. Indeed, if $d<m$, we have $\mathrm{Nef}(Y)=\R_{\geq 0} H_2+ \R_{\geq 0} H_1$, and the line bundles $p_2^*L'_{|Y},\; p_1^*\mathscr{O}_{\P^1}(1)_{|Y}$ are semiample. This is enough to conclude that $Y$ is a Mori dream space.
\end{proof}

\begin{proofB}
It is a direct consequence of Proposition \ref{prop:birgeom} and Corollary \ref{cor}.

\end{proofB}

\section{The Cox ring of the hypersurfaces}
In this section we interpret algebraically the results obtained in the previous section. In particular, we compute the Cox rings of the hypersurface we have considered. To this end, we start by recalling the construction in~\cite{hlu}. 
\medskip

Let $Z$ be a normal projective variety of dimension at least $3$. We assume
that $\operatorname{Cl}(Z)$ has rank $1$ and that its Cox ring
$R:=\mathcal R(Z)$ is finitely generated. 
We denote by $\overline Z:=\operatorname{Spec}(R)$ the total coordinate space
of $Z$, and by $\widehat Z\subseteq \overline Z$ its characteristic space.
Since the irrelevant ideal is the radical of the ideal generated by the
homogeneous elements of positive degree, the complement
$\overline Z\setminus \widehat Z$ is the vertex of the affine cone
$\overline Z$. In particular, its codimension in $\overline Z$ is
$\dim Z+1$, hence at least $4$.
The Cox ring of $X:=\mathbf P^1\times Z$ is naturally identified with
$R[T_1,T_2]$, where $T_1$ and $T_2$ are the Cox coordinates of the
$\mathbf P^1$-factor. Its total coordinate space is
$\mathbf A^2\times \overline Z$, whereas its characteristic space is
$(\mathbf A^2\setminus\{0\})\times \widehat Z$. The quotient map fits into
the diagram
\[
\begin{tikzcd}
(\mathbf A^2\setminus\{0\})\times \widehat Z
\arrow[r,hook]
\arrow[d,"p"']
&
\mathbf A^2\times \overline Z
\\
\mathbf P^1\times Z.
\end{tikzcd}
\]
The irrelevant locus in $\mathbf A^2\times \overline Z$ is
$V(T_1,T_2)\cup \bigl(\mathbf A^2\times(\overline Z\setminus \widehat Z)\bigr)$.
The second component has codimension at least $4$ in
$\mathbf A^2\times \overline Z$. Hence the only irreducible component of the
irrelevant locus of codimension $2$ is $V(T_1,T_2)$.
Let now $Y \subseteq \mathbf P^1\times Z$ be a normal hypersurface defined by
the homogeneous equation
\[
f=
f_0T_1^d-f_1T_1^{d-1}T_2+\cdots+(-1)^df_dT_2^d,
\tag{1}
\]
where $d\geq 1$ and $f_0,\dots,f_d\in R$ are homogeneous elements of the same
degree. Set $A:=R[T_1,T_2]/\langle f\rangle$.
Let $Y_0:=Y\cap X_{\mathrm{reg}}$.
Since $Y$ is normal, $Y_0$ is a big open subset of $Y$. Indeed, suppose that
a codimension-one point $y\in Y$ lies in the singular locus of $X$. Since
$Y$ is normal, $\mathcal O_{Y,y}$ is regular. On the other hand, $Y$ is
locally cut out in $X$ by one non-zero divisor. Thus $\mathcal O_{Y,y}$ is a
hypersurface quotient of $\mathcal O_{X,y}$, and the regularity of
$\mathcal O_{Y,y}$ forces $\mathcal O_{X,y}$ to be regular. This contradicts
the assumption that $x$ lies in the singular locus of $X$. Hence
$Y\setminus Y_0$ has codimension at least $2$ in $Y$.
Define $\widehat Y$ as the Zariski closure of $p^{-1}(Y_0)$ inside
$(\mathbf A^2\setminus\{0\})\times\widehat Z$, and define $\widetilde Y$ as
the Zariski closure of $p^{-1}(Y_0)$ inside $\mathbf A^2\times\overline Z$.
Then $\widetilde Y=V(f)$. Moreover, since the only codimension-two component
of the irrelevant locus is $V(T_1,T_2)$, the open subset obtained from
$\widetilde Y$ by removing the codimension-one components of
$\widetilde Y\setminus\widehat Y$ is
$\widetilde Y_1=V(f)\setminus V(T_1,T_2)$.
We shall use the following algebraic description of
$A_{T_1}\cap A_{T_2}$ inside $A_{T_1T_2}$.
\medskip

\begin{prop}\label{prop:intersection}
Let \(R\) be a noetherian integral domain, and let
\(f_0,\dots,f_d\in R\). Set
\[
f=f_0T_1^d-f_1T_1^{d-1}T_2+\cdots+(-1)^df_dT_2^d
\in R[T_1,T_2],
\]
and let \(A:=R[T_1,T_2]/\langle f\rangle\).
Let \(B:=R[T_1,T_2,S_1,\dots,S_d]\), and let
\(I:=\langle g_0,\dots,g_d\rangle\), where
\[
g_0=f_0+T_2S_1,\qquad
g_i=f_i+T_1S_i+T_2S_{i+1},\qquad
g_d=f_d+T_1S_d.
\]
Here \(1\leq i\leq d-1\) in the middle formula. Assume that
\(f_0,\dots,f_d\) is a regular sequence in \(R\). Then the natural map
\(B/I\to A_{T_1T_2}\) identifies \(B/I\) with
\(A_{T_1}\cap A_{T_2}\) inside \(A_{T_1T_2}\).
\end{prop}

\begin{proof}
Set \(C:=B/I\). We first describe the map to \(A_{T_1T_2}\).
It is induced by the identity on \(R[T_1,T_2]\) and by the recursive
assignments
\[
S_1=-\frac{f_0}{T_2},\qquad
S_{i+1}=-\frac{f_i+T_1S_i}{T_2}
\quad (1\leq i\leq d-1).
\]
The relations \(g_0,\dots,g_d\) are satisfied in \(A_{T_1T_2}\), so this gives
a well-defined homomorphism \(C\to A_{T_1T_2}\).
After localizing at \(T_2\), the relations \(g_0,\dots,g_{d-1}\) eliminate
\(S_1,\dots,S_d\) recursively. With these substitutions, the last relation
\(g_d\) becomes \((-1)^df/T_2^d\). Hence \(C_{T_2}\cong A_{T_2}\). Similarly,
after localizing at \(T_1\), one eliminates the variables in the opposite
direction, starting from \(S_d=-f_d/T_1\), and the remaining relation becomes
\(f/T_1^d\). Hence \(C_{T_1}\cong A_{T_1}\). These identifications are
compatible after localizing at \(T_1T_2\), and therefore
\(C_{T_1T_2}\cong A_{T_1T_2}\).
It remains to prove that \(C=C_{T_1}\cap C_{T_2}\) inside \(C_{T_1T_2}\).
We show that \(C\) has depth at least \(2\) along
\(\langle T_1,T_2\rangle C\).
Since \(T_1,T_2\) is a regular sequence in \(B\), it is enough to look modulo
\(\langle T_1,T_2\rangle\). In that quotient, the images of
\(g_0,\dots,g_d\) are exactly \(f_0,\dots,f_d\) in
\(R[S_1,\dots,S_d]\). Since polynomial extension is flat,
\(f_0,\dots,f_d\) remains a regular sequence in \(R[S_1,\dots,S_d]\). Hence
\(T_1,T_2,g_0,\dots,g_d\) is a regular sequence in \(B\).
Let \(\mathfrak p\) be a prime ideal of \(B\) containing
\(I+\langle T_1,T_2\rangle\). Localizing at \(\mathfrak p\), the sequence
\(T_1,T_2,g_0,\dots,g_d\) remains regular. Since regular sequences are
permutable in local rings, \(g_0,\dots,g_d,T_1,T_2\) is also a regular sequence
in \(B_{\mathfrak p}\). Therefore the images of \(T_1,T_2\) form a regular
sequence on \(C_{\mathfrak p}\). This proves that
\[
\operatorname{depth}_{\langle T_1,T_2\rangle C}(C)\geq 2.
\]
Consequently the local cohomology groups
\(H^0_{\langle T_1,T_2\rangle C}(C)\) and
\(H^1_{\langle T_1,T_2\rangle C}(C)\) vanish. Equivalently, the
\(\check{\mathrm C}\)ech sequence associated with the cover
\(D(T_1)\cup D(T_2)\) is exact:
\[
0\longrightarrow C
\longrightarrow C_{T_1}\oplus C_{T_2}
\longrightarrow C_{T_1T_2}.
\]
Thus \(C=C_{T_1}\cap C_{T_2}\) inside \(C_{T_1T_2}\). Using the
identifications \(C_{T_1}\cong A_{T_1}\), \(C_{T_2}\cong A_{T_2}\), and
\(C_{T_1T_2}\cong A_{T_1T_2}\), we obtain
\[
C\cong A_{T_1}\cap A_{T_2}
\]
inside \(A_{T_1T_2}\), as claimed.
\end{proof}

We can now prove Theorem \ref{thm:p1z-cox}.

\begin{proofA}
By assumption, the open subset
$\widetilde Y_1=V(f)\setminus V(T_1,T_2)$ is normal. Moreover, the pullback on
divisor class groups is an isomorphism. Hence~\cite[Thm. 1]{hlu} applied
to the embedding $Y\subseteq \mathbf P^1\times Z$ gives
\[
\mathcal R(Y)=A_{T_1}\cap A_{T_2}
\quad\text{inside } A_{T_1T_2}.
\tag{6}
\]
Here we use that the only codimension-two component of the irrelevant locus
of $\mathbf P^1\times Z$ is $V(T_1,T_2)$.
By Proposition~\ref{prop:intersection}, the regularity of the sequence
$f_0,\dots,f_d$ identifies the intersection in \textup{(6)} with the quotient
displayed in \textup{(5)}. This proves the theorem.
\end{proofA}

\begin{proof}[Proof of Corollary~\ref{cor:general-p1z-cox}]
Since \(Y\) is a member of the complete linear system associated with the
invertible sheaf \(L\), it is a Cartier divisor on
\(X=\mathbf P^1\times Z\). Thus the first hypothesis of
Theorem~\ref{thm:p1z-cox} is satisfied.
We now prove the normality of the open subset
\[
\widetilde Y_1:=V(f)\setminus V(T_1,T_2)
\subseteq \mathbf A^2\times \overline Z.
\]
Set \(W:=(\mathbf A^2\times \overline Z)\setminus V(T_1,T_2)\).
Since \(\overline Z\) is Cohen--Macaulay, the open subset \(W\) is
Cohen--Macaulay. The scheme \(\widetilde Y_1\) is cut out in \(W\) by the
single equation \(f=0\). Moreover, \(W\) is integral and the restriction of
\(f\) to \(W\) is nonzero; hence \(f\) is a non-zero divisor in every local
ring of \(W\). Therefore \(\widetilde Y_1\) is an effective Cartier divisor on
\(W\). In particular, \(\widetilde Y_1\) is Cohen--Macaulay, and hence it
satisfies Serre's condition \(S_2\).\\

Let \(Y_0:=Y\cap X_{\mathrm{reg}}\), and let \(\widehat Y\) be the Zariski
closure of \(p^{-1}(Y_0)\) in the characteristic space
\((\mathbf A^2\setminus\{0\})\times \widehat Z\). By
\cite[Prop.~1.2]{hlu}, for a general choice of \(Y\), the variety
\(\widehat Y\) is normal. Moreover, by construction of \(\widetilde Y_1\),
the complement \(\widetilde Y_1\setminus \widehat Y\) has codimension at least
\(2\) in \(\widetilde Y_1\). Thus every codimension-one point of
\(\widetilde Y_1\) lies on \(\widehat Y\). Since \(\widehat Y\) is normal, it
is regular in codimension one, and consequently \(\widetilde Y_1\) satisfies
Serre's condition \(R_1\). Hence \(\widetilde Y_1\) is normal by Serre's
criterion. This proves the second hypothesis of Theorem~\ref{thm:p1z-cox}.\\

Next, since \(L\) is ample, base-point free, and
\(\dim X=\dim Z+1\geq 4\), the Grothendieck--Lefschetz theorem for class
groups of Ravindra--Srinivas applies to the general member \(Y\in|L|\).
Therefore, the restriction map
\[
\iota^*\colon \operatorname{Cl}(X)\xrightarrow{\sim}\operatorname{Cl}(Y)
\]
is an isomorphism. This gives the third hypothesis of
Theorem~\ref{thm:p1z-cox}.
Finally, the fourth hypothesis of Theorem~\ref{thm:p1z-cox} is exactly the
assumption that \(f_0,\dots,f_d\) is a regular sequence in \(R\). Hence all
hypotheses of Theorem~\ref{thm:p1z-cox} are satisfied, and the stated
presentation of \(\mathcal R(Y)\) follows.
\end{proof}

\begin{proof}[Proof of Corollary~\ref{cor:Fano}]
By \cite[Theorem 1.2]{Brown13}, \cite[Theorem 1.1]{GOST15} and \cite[Remark 2.5]{KO15}, the ring $\mathcal{R}(Z)$ is Cohen--Macaulay, hence the result follows by Corollary \ref{cor:general-p1z-cox}.
\end{proof}

We conclude this note with an explicit example.

\begin{ex}\label{ex:cy-p1-complexity-one}
Let \(Z\) be the Fano threefold No.~1 in~\cite[Thm. 1.1]{bhhn}. Thus
\(\operatorname{Cl}(Z)\cong \mathbf Z\), all Cox variables have degree \(1\),
and
\[
\mathcal R(Z)=
\frac{\mathbf K[T_1,\dots,T_5]}
{\langle T_1T_2+T_3T_4+T_5^2\rangle}.
\]
This variety is Gorenstein, and its anticanonical class is \(-K_Z=3H\), where
\(H\) denotes the positive generator of \(\operatorname{Cl}(Z)\).
Consider \(X:=\mathbf P^1\times Z\). We denote by \(T_6,T_7\) the Cox
coordinates of the \(\mathbf P^1\)-factor. Then
\(\mathcal R(X)=\mathcal R(Z)[T_6,T_7]\). If \(H_1\) denotes the pullback of
\(\mathcal O_{\mathbf P^1}(1)\) and \(H_2\) the pullback of \(H\), then
\(-K_X=2H_1+3H_2\).
Let \(Y\subseteq X\) be a general anticanonical hypersurface. Its equation has
the form
\[
f=f_0T_6^2-f_1T_6T_7+f_2T_7^2,
\]
where \(f_0,f_1,f_2\in \mathcal R(Z)_3\) are general homogeneous elements of
degree \(3\). By adjunction, \(K_Y\sim 0\), so \(Y\) is a Calabi--Yau
threefold, provided it is normal.
Since \(\mathcal R(Z)\) is a Cohen--Macaulay domain and the elements
\(f_0,f_1,f_2\) are general of positive degree, they form a regular sequence.
Thus Theorem~\ref{thm:p1z-cox} applies and gives
\[
\mathcal R(Y)\cong
\frac{
\mathcal R(Z)[T_6,T_7,S_1,S_2]
}{
\langle
f_0+T_7S_1,\,
f_1+T_6S_1+T_7S_2,\,
f_2+T_6S_2
\rangle}.
\]
Equivalently, replacing \(\mathcal R(Z)\) by its explicit presentation, this is
\[
\mathcal R(Y)\cong
\frac{
\mathbf K[T_1,\dots,T_7,S_1,S_2]
}{
\langle
T_1T_2+T_3T_4+T_5^2,\,
f_0+T_7S_1,\,
f_1+T_6S_1+T_7S_2,\,
f_2+T_6S_2
\rangle}.
\]
The grading is induced by the isomorphism
\(\operatorname{Cl}(X)\cong \operatorname{Cl}(Y)\). In particular, if
\(\deg(T_6)=\deg(T_7)=(1,0)\) and \(\deg(T_j)=(0,1)\) for \(1\leq j\leq 5\),
then \(\deg(S_1)=\deg(S_2)=(-1,3)\).
\end{ex}

\bibliographystyle{plain}
\bibliography{ref}

@book {adhl,
    AUTHOR = {Arzhantsev, Ivan and Derenthal, Ulrich and Hausen, J\"{u}rgen
              and Laface, Antonio},
     TITLE = {Cox rings},
    SERIES = {Cambridge Studies in Advanced Mathematics},
    VOLUME = {144},
 PUBLISHER = {Cambridge University Press, Cambridge},
      YEAR = {2015},
     PAGES = {viii+530},
      ISBN = {978-1-107-02462-5},
   MRCLASS = {14Cxx (14Jxx 14Lxx)},
  MRNUMBER = {3307753},
MRREVIEWER = {Alexandr\ V.\ Pukhlikov},
}

@misc{hlu,
      title={The {C}ox ring of an embedded variety}, 
      author={Cristóbal Herrera and Antonio Laface and Luca Ugaglia},
      year={2024},
      eprint={2411.17370},
      archivePrefix={arXiv},
      primaryClass={math.AG},
      url={https://arxiv.org/abs/2411.17370}, 
}

@article {Ottem15,
    AUTHOR = {Ottem, John C.},
     TITLE = {Birational geometry of hypersurfaces in products of projective
              spaces},
   JOURNAL = {Math. Z.},
    VOLUME = {280},
      YEAR = {2015},
    NUMBER = {1-2},
     PAGES = {135--148}
       }

@misc{PIS25,
      title={On the {C}ox rings of some hypersurfaces}, 
      author={Andrew Pollock and Atsushi Ito and Balazs Szendroi},
      year={2025},
      eprint={2504.07591},
      archivePrefix={arXiv},
      primaryClass={math.AG},
      url={https://arxiv.org/abs/2504.07591}, 
}

@article{EGA4.3,
     author = {Grothendieck, Alexander},
     title = {El\'ements de g\'eom\'etrie alg\'ebrique : {IV.} {Etude} locale des sch\'emas et des morphismes de sch\'emas, {Troisi\`eme} partie},
     journal = {Inst. Hautes Études Sci. Publ. Math.},
     pages = {5--255},
     publisher = {Inst. Hautes Études Sci.
Presses Univ. France},
     volume = {28},
     year = {1966},
     url = {http://www.numdam.org/item/PMIHES_1966__28__5_0/}
}

@article {BCHM10,
    AUTHOR = {Birkar, Caucher and Cascini, Paolo and Hacon, Christopher D.
              and McKernan, James},
     TITLE = {Existence of minimal models for varieties of log general type},
   JOURNAL = {J. Amer. Math. Soc.},
    VOLUME = {23},
      YEAR = {2010},
    NUMBER = {2},
     PAGES = {405--468}
}

@book{Ott95,
  title={Variet{\`a} proiettive di codimensione piccola},
  author={Ottaviani, Giorgio},
  series={Ist. nazion. di alta matematica F. Severi},
  url={https://books.google.fr/books?id=HmHqngEACAAJ},
  year={1995},
  publisher={Aracne}
}

@book{ACGH13,
  title={Geometry of Algebraic Curves: Volume I},
  author={Arbarello, Enrico and Cornalba, Maurizio and Griffiths, Phillip and Harris, Joseph D.},
  series={Grundlehren der mathematischen Wissenschaften},
  year={2013},
  publisher={Springer New York}
}

@article {KO15,
    AUTHOR = {Kawamata, Yujiro and Okawa, Shinnosuke},
     TITLE = {Mori dream spaces of {C}alabi-{Y}au type and log canonicity of
              {C}ox rings},
   JOURNAL = {J. Reine Angew. Math.},
  FJOURNAL = {Journal f\"ur die Reine und Angewandte Mathematik. [Crelle's
              Journal]},
    VOLUME = {701},
      YEAR = {2015},
     PAGES = {195--203}
}

@article {GOST15,
    AUTHOR = {Gongyo, Yoshinori and Okawa, Shinnosuke and Sannai, Akiyoshi
              and Takagi, Shunsuke},
     TITLE = {Characterization of varieties of {F}ano type via singularities
              of {C}ox rings},
   JOURNAL = {J. Algebraic Geom.},
    VOLUME = {24},
      YEAR = {2015},
    NUMBER = {1},
     PAGES = {159--182}
}

@article {Brown13,
    AUTHOR = {Brown, Morgan},
     TITLE = {Singularities of {C}ox rings of {F}ano varieties},
   JOURNAL = {J. Math. Pures Appl. (9)},
    VOLUME = {99},
      YEAR = {2013},
    NUMBER = {6},
     PAGES = {655--667}
}

@article {Ito14,
    AUTHOR = {Ito, Atsushi},
     TITLE = {Examples of {M}ori dream spaces with {P}icard number two},
   JOURNAL = {Manuscripta Math.},
    VOLUME = {145},
      YEAR = {2014},
    NUMBER = {3-4},
     PAGES = {243--254}
}

@incollection{HK00,
    AUTHOR = {Hu, Yi and Keel, Sean},
     TITLE = {Mori dream spaces and {GIT}},
      NOTE = {Dedicated to William Fulton on the occasion of his 60th
              birthday},
   JOURNAL = {Michigan Math. J.},
  FJOURNAL = {Michigan Mathematical Journal},
    VOLUME = {48},
      YEAR = {2000},
     PAGES = {331--348}
}

@misc{SP24,
  author       = {The {Stacks project authors}},
  title        = {The Stacks project},
  howpublished = {\url{https://stacks.math.columbia.edu}},
  year         = {2024},
}

@book{Fulton98,
  title={Intersection theory},
  author={Fulton, William},
  year={1998},
  edition={2nd},
  publisher={Springer Science \& Business Media},
  address={Berlin, New York},
  isbn={978-0-387-98549-7}
}

@article {RS06,
    AUTHOR = {Ravindra, Girivaru V. and Srinivas, Vasudevan},
     TITLE = {The {G}rothendieck-{L}efschetz theorem for normal projective
              varieties},
   JOURNAL = {J. Algebraic Geom.},
    VOLUME = {15},
      YEAR = {2006},
    NUMBER = {3},
     PAGES = {563--590}
}

@book{Matsumura87, 
place={Cambridge}, 
series={Cambridge Studies in Advanced Mathematics}, 
title={Commutative Ring Theory}, 
publisher={Cambridge University Press}, 
author={Matsumura, Hideyuki}, 
editor={Reid, Miles}, 
year={1987}, 
collection={Cambridge Studies in Advanced Mathematics}
}

@article{Denisi24,
      title={Birational geometry of hypersurfaces in products of weighted projective spaces}, 
      author={Francesco Antonio Denisi},
      journal = {Math. Z.},
      year={2026},
      volume = {313},
      number = {2},
      doi={10.1007/s00209-026-04023-6}
}

@article{bhhn,
    author = {Bechtold, Benjamin and Hausen, Jürgen and Huggenberger, Elaine and Nicolussi, Michele},
    title = {On Terminal Fano 3-Folds with 2-Torus Action},
    journal = {Int. Math. Res. Not. IMRN},
    volume = {2016},
    number = {5},
    pages = {1563-1602},
    year = {2016},
    month = {01},
    issn = {1073-7928},
    doi = {10.1093/imrn/rnv190},
    url = {https://doi.org/10.1093/imrn/rnv190},
    eprint = {https://academic.oup.com/imrn/article-pdf/2016/5/1563/7374860/rnv190.pdf},
}

\end{document}